\documentclass[a4paper,10pt]{article}

\usepackage[square,numbers]{natbib}
\usepackage{amsfonts,amsmath,amssymb,amsthm}
\usepackage{hyperref}
\usepackage{graphicx}
\usepackage{mathpazo}
\usepackage[font={small}]{caption}
\usepackage{tikz}
\usepackage{tikz-cd}
\usepackage{float}
\usepackage{booktabs}
\usepackage{xargs}
\usepackage{MnSymbol}

\usetikzlibrary{arrows}
\usetikzlibrary{decorations}
\usetikzlibrary{shapes}
\usetikzlibrary{positioning}
\usetikzlibrary{decorations.pathmorphing,shapes}
\usetikzlibrary{backgrounds}

\allowdisplaybreaks
\theoremstyle{definition}
\newtheorem{theorem}{Theorem}[section]

\newtheorem{prop}[theorem]{Proposition}

\newtheorem{cor}[theorem]{Corollary}

\newtheorem{rem}[theorem]{Remark}
\newtheorem{ex}[theorem]{Example}

\newcommand*{\scalingFactorSmall}{0.4}

\DeclareMathOperator{\rk}{rk}

\DeclareMathOperator{\id}{id}

\title{Maximum antichains in posets of quiver representations}
\author{Florian Gellert, Philipp Lampe}

\begin{document}


  \makeatletter
  \pgfdeclareshape{rectangle with diagonal fill}
  {
      \inheritsavedanchors[from=rectangle]
      \inheritanchorborder[from=rectangle]
      \inheritanchor[from=rectangle]{north}
      \inheritanchor[from=rectangle]{north west}
      \inheritanchor[from=rectangle]{north east}
      \inheritanchor[from=rectangle]{center}
      \inheritanchor[from=rectangle]{west}
      \inheritanchor[from=rectangle]{east}
      \inheritanchor[from=rectangle]{mid}
      \inheritanchor[from=rectangle]{mid west}
      \inheritanchor[from=rectangle]{mid east}
      \inheritanchor[from=rectangle]{base}
      \inheritanchor[from=rectangle]{base west}
      \inheritanchor[from=rectangle]{base east}
      \inheritanchor[from=rectangle]{south}
      \inheritanchor[from=rectangle]{south west}
      \inheritanchor[from=rectangle]{south east}

      \inheritbackgroundpath[from=rectangle]
      \inheritbeforebackgroundpath[from=rectangle]
      \inheritbehindforegroundpath[from=rectangle]
      \inheritforegroundpath[from=rectangle]
      \inheritbeforeforegroundpath[from=rectangle]

      \behindbackgroundpath{%
          \pgfextractx{\pgf@xa}{\southwest}%
          \pgfextracty{\pgf@ya}{\southwest}%
          \pgfextractx{\pgf@xb}{\northeast}%
          \pgfextracty{\pgf@yb}{\northeast}%
          \ifpgf@diagonal@lefttoright
              \def\pgf@diagonal@point@a{\pgfpoint{\pgf@xa}{\pgf@yb}}%
              \def\pgf@diagonal@point@b{\pgfpoint{\pgf@xb}{\pgf@ya}}%
          \else
              \def\pgf@diagonal@point@a{\southwest}%
              \def\pgf@diagonal@point@b{\northeast}%
          \fi
          \pgfpathmoveto{\pgf@diagonal@point@a}%
          \pgfpathlineto{\northeast}%
          \pgfpathlineto{\pgfpoint{\pgf@xb}{\pgf@ya}}%
          \pgfpathclose
          \ifpgf@diagonal@lefttoright
              \color{\pgf@diagonal@top@color}%
          \else
              \color{\pgf@diagonal@bottom@color}%
          \fi
          \pgfusepath{fill}%
          \pgfpathmoveto{\pgfpoint{\pgf@xa}{\pgf@yb}}%
          \pgfpathlineto{\southwest}%
          \pgfpathlineto{\pgf@diagonal@point@b}%
          \pgfpathclose
          \ifpgf@diagonal@lefttoright
              \color{\pgf@diagonal@bottom@color}%
          \else
              \color{\pgf@diagonal@top@color}%
          \fi
          \pgfusepath{fill}%
      }
  }

  \newif\ifpgf@diagonal@lefttoright
  \def\pgf@diagonal@top@color{white}
  \def\pgf@diagonal@bottom@color{gray!30}

  \def\pgfsetdiagonaltopcolor#1{\def\pgf@diagonal@top@color{#1}}%
  \def\pgfsetdiagonalbottomcolor#1{\def\pgf@diagonal@bottom@color{#1}}%
  \def\pgfsetdiagonallefttoright{\pgf@diagonal@lefttorighttrue}%
  \def\pgfsetdiagonalrighttoleft{\pgf@diagonal@lefttorightfalse}%

  \tikzoption{diagonal top color}{\pgfsetdiagonaltopcolor{#1}}
  \tikzoption{diagonal bottom color}{\pgfsetdiagonalbottomcolor{#1}}
  \tikzoption{diagonal from left to right}[]{\pgfsetdiagonallefttoright}
  \tikzoption{diagonal from right to left}[]{\pgfsetdiagonalrighttoleft}

  \makeatother

  \maketitle

  \abstract{
  We study maximum antichains in two posets related to quiver representations. Firstly, we consider the set of isomorphism classes of indecomposable representations ordered by inclusion. For various orientations of the Dynkin diagram of type A we construct a maximum antichain in the poset.
  Secondly, we consider the set of subrepresentations of a given quiver representation, again ordered by inclusion. It is a finite set if we restrict to linear representations over finite fields or to representations with values in the category of pointed sets. For particular situations we prove that this poset is Sperner.
  }

  \section{Introduction and notation}

  \subsection{Maximum antichains in posets}

  Let $(P,\leq)$ be a poset. Two elements $a,b\in P$ are called \textit{incomparable} if neither $a\leq b$ nor $b\leq a$ holds. The elements are called \textit{comparable} otherwise. A subset $\mathcal{F}\subseteq P$ of pairwise incomparable elements is called an \textit{antichain}. An antichain $\mathcal{F}\subseteq P$ is called \textit{maximal} if there does not exist an element $a\in P$ such that $\mathcal{F}\cup\{a\}$ is an antichain. It is called \textit{maximum} if there does not exist an antichain $\mathcal{F}'\subseteq P$ such that $\vert \mathcal{F}'\vert>\vert \mathcal{F}\vert$. Note that every maximum antichain is a maximal antichain, but the converese does not hold in general. The size of maximum antichain is sometimes called the \textit{width} of the poset. Furthermore, a subset $C\subseteq P$ of pairwise comparable elements is called a \textit{chain}. Note that the elements of a chain can be reordered to form a sequence $(a_1\leq a_2\leq\ldots\leq a_k)$ and we will often use this notation to describe a chain. \textit{Maximal} and \textit{maximum} chains are defined in a similar way to maximal and maximum antichains.

  Let $n\geq 1$ be an integer. We denote by $\mathcal{P}_n$ set of all subsets of the finite set $\{1,2,\ldots,n\}$. Note that $\mathcal{P}_n$ is partially ordered by inclusion. Sperner \cite{Sp} constructs a maximum antichain in $(\mathcal{P}_n,\subseteq)$:

  \begin{theorem}[Sperner]
  The set $\{A\in\mathcal{P}_n\colon \vert A\vert=\lfloor n/2\rfloor\}$ is a maximum antichain in the poset $(\mathcal{P}_n,\subseteq)$ so that width of the poset is given by the binomial coefficient $\dbinom{n}{\lfloor n/2\rfloor}$.
  \end{theorem}

  In a later work, Stanley \cite[Theorem 2.2]{St} gives an elegant proof of Sperner's theorem using linear algebra and a grading of the poset. We give a sketch of the proof after recalling some basic notions about posets. We say that the element $a\in P$ \textit{covers} the element $b\in P$ if $a>b$ and there does not exist an element $c\in P$ with $a>c>b$. Moreover, we say that $\textit{a}\in P$ is \textit{minimal} if there does not exist an element $b\in P$ with $b<a$. Dually, we say that $\textit{a}\in P$ is \textit{maximal} if there does not exist an element $b\in P$ with $b>a$. The poset $(P,\leq)$ is called \textit{graded} if there exists a map $\operatorname{deg}\colon P\to\mathbb{N}$ such that $\operatorname{deg}(a)=0$ for all minimal elements $a\in P$ and $\operatorname{deg}(a)=\operatorname{deg}(b)+1$ whenever $a$ covers $b$. For example, the poset $\mathcal{P}_n$ is graded by the cardinality viewed as a map $\vert\cdot\vert\colon \mathcal{P}_n\to\mathbb{N}$. A poset $(P,\leq)$ is called \textit{bounded} if it contains a \textit{miminum}, i.\,e. an element $a\in P$ such that $a\leq b$ for every $b\in P$, and a \textit{maximum}, i.\,e. an element $a\in P$ such that $b\leq a$ for every $b\in P$. Note that a finite, bounded poset is graded if and only if all maximal chains have the same cardinality. The rank of a graded poset $(P,\leq)$ is $\rk(P)=\operatorname{max}(\operatorname{deg}(a)\colon a\in P)$.

  Suppose that the poset $(P,\leq)$ is indeed finite, bounded and graded of rank $n$. Then we denote by $P_i\subseteq P$ the subset of elements of degree $i$. It is easy to see that every $P_i$ with $i\geq 0$ is an antichain. We say that $P$ is \textit{Sperner} if there exists a natural number $i$ such that $\vert\mathcal{F}\vert\leq \vert P_i\vert$ for every antichain $\mathcal{F}\subseteq P$. In other words, $P_i$ is a maximum antichain but there may exist other maximum antichains in $P$. In the case of the power set $\mathcal{P}_n$ we simply write $\mathcal{P}_{n,i}$ instead of $(\mathcal{P}_n)_i$. Sperner's theorem implies that the power set $\mathcal{P}_n$ is Sperner. For every $i$ we consider the $\mathbb{Q}$-vector space $\mathbb{Q}P_i$ with basis $P_i$.

  \begin{theorem}[Stanley]
  \label{thm:Stanley}
  Let $k$ be a natural number. Suppose that there are injective linear maps \linebreak $U_i\colon \mathbb{Q}P_i\to\mathbb{Q}P_{i+1}$ for $0\leq i\leq k-1$ such that $U_i(a)\in\sum_{b\geq a}\mathbb{Q}b$ for every $a\in P_i$; suppose further that there are injective linear maps $D_i\colon\mathbb{Q}P_{i}\to\mathbb{Q}P_{i-1}$ for $k+1\leq i\leq n$ such that for $D_i(a)\in\sum_{b<a}\mathbb{Q}b$ for every $a\in P_i$, \linebreak see Figure~\ref{fig:stanley} for an illustration. Then the poset $(P,\leq)$ is Sperner and $P_k$ is a maximum antichain.
  \begin{figure}
	  \begin{center}
		  \begin{tikzpicture}
			  \node (1) at (0,0) {$\mathbb{Q}P_0$};
			  \node (2) at (2,0) {$\mathbb{Q}P_1$};
			  \node (3) at (4,0) {$\mathbb{Q}P_2$};
			  \node (4) at (6,0) {$\cdots$};
			  \node (5) at (8,0) {$\mathbb{Q}P_{k-1}$};
			  \node (6) at (10,0) {$\mathbb{Q}P_k$};
			  \node (7) at (0,-0.8) {$\mathbb{Q}P_k$};
			  \node (8) at (2,-0.8) {$\mathbb{Q}P_{k+1}$};
			  \node (9) at (4,-0.8) {$\mathbb{Q}P_{k+2}$};
			  \node (10) at (6,-0.8) {$\cdots$};
			  \node (11) at (8,-0.8) {$\mathbb{Q}P_{n-1}$};
			  \node (12) at (10,-0.8) {$\mathbb{Q}P_n$};

			  \draw[right hook-latex,above] (1) to node {\footnotesize{$U_0$}} (2);
			  \draw[right hook-latex,above] (2) to node {\footnotesize{$U_1$}} (3);
			  \draw[right hook-latex,above] (3) to node {\footnotesize{$U_2$}} (4);
			  \draw[right hook-latex,above] (4) to node {\footnotesize{$U_{k-2}$}} (5);
			  \draw[right hook-latex,above] (5) to node {\footnotesize{$U_{k-1}$}} (6);
			  \draw[left hook-latex,>= latex,above] (8) to node {\footnotesize{$D_{k+1}$}} (7);
			  \draw[left hook-latex,>= latex,above] (9) to node {\footnotesize{$D_{k+2}$}} (8);
			  \draw[left hook-latex,>= latex,above] (10) to node {\footnotesize{$D_{k+3}$}} (9);
			  \draw[left hook-latex,>= latex,above] (11) to node {\footnotesize{$D_{n-1}$}} (10);
			  \draw[left hook-latex,>= latex,above] (12) to node {\footnotesize{$D_{n}$}} (11);
		  \end{tikzpicture}
	  \end{center}
	  \caption{Injective maps in Stanley's theorem}\label{fig:stanley}
  \end{figure}
  \end{theorem}

  As a first application, Stanley gives an elegant proof of Sperner's theorem. A crucial role in the proof is played by the linear maps $U,D\colon \mathbb{Q}\mathcal{P}_n\to\mathbb{Q}\mathcal{P}_n$ defined by the formulae $U(A)=\sum_{b\notin A}A\cup \{b\}$\linebreak and $D(A)=\sum_{a\in A} A\setminus \{a\}$ for all $A\in\mathcal{P}_n$. Moreover, we denote by $U_i\colon\mathbb{Q}\mathcal{P}_{n,i}\to\mathbb{Q}\mathcal{P}_{n,i+1}$ and \linebreak $D_i\colon\mathbb{Q}\mathcal{P}_{n,i}\to\mathbb{Q}\mathcal{P}_{n,i-1}$ the restrictions of $U$ and $D$ to the homogeneous component of degree $i$. An essential observation in Stanley's argument is the commutation relation $D_{i+1}U_i-U_{i-1}D_i=(n-2i)\operatorname{id}_{\mathbb{Q}P_{n,i}}$ for all $i$. The relation implies that $D_{i+1}U_i$ is positive definite for $0\leq i< \left\lfloor\tfrac n2\right\rfloor$. Hence $D_{i+1}U_i$ is invertible so that $U_i$ must be injective in this case. By a similar argument we can show that $D_i$ is injective for $\left\lfloor\tfrac n2\right\rfloor< i\leq n$. As a second application, Stanley proves that the poset of vector subspaces of a finite-dimensional vector space over a finite field, again ordered by inclusion, is Sperner.

  A \textit{chain decomposition} of a poset $(P,\leq)$ is a disjoint union $P=C_1\sqcup C_2\sqcup\ldots \sqcup C_k$ where every $C_i$ is a chain in $P$. If $\mathcal{F}\subseteq P$ is an antichain, then every chain $C_i$ contains at most one element of $\mathcal{F}$. Especially, we have $\vert\mathcal{F}\vert\leq k$. A chain decomposition is called a \textit{Dilworth decomposition} if there does not exist a chain decomposition with a smaller number of chains. The next theorem is due to Dilworth \cite[Theorem 1.1]{D}:

  \begin{theorem}[Dilworth] The cardinality of a maximum antichain in a finite poset $(P,\leq)$ is equal to the smallest number of chains in a chain decomposition of $(P,\leq)$.
  \end{theorem}

  A chain $C=(a_1\leq a_2\leq\ldots\leq a_k)$ in $P$ is called \textit{saturated} if $a_{i+1}$ covers $a_i$ for all $1\leq i\leq k-1$. Suppose that the poset $(P,\leq)$ is graded with degree map $\operatorname{deg}\colon P\to\mathbb{N}$ and has finite rank $n=\rk(P)$. A chain $C=(a_1\leq a_2\leq\ldots\leq a_k)$ in $P$ is called \textit{symmetric} if it is saturated and the equality $\operatorname{deg}(a_1)+\operatorname{deg}(a_k)=n$ holds. A chain decomposition $P=C_1\sqcup C_2\sqcup\ldots \sqcup C_k$ is called \textit{symmetric} if every chain $C_i$ is symmetric. We say that the graded poset $(P,\leq)$ is a \textit{symmetric chain order} if it admits a symmetric chain decomposition. Engel \cite[Theorem 5.1.4]{E} proves a relationship between symmetric chain orders and Sperner posets:

  \begin{theorem}
  If the graded poset $(P,\leq)$ is a symmetric chain order, then it is Sperner.
  \end{theorem}

  As an application of the theorem, M\"uhle \cite{M} shows that certain posets of noncrossing partitions are Sperner. Noncrossing partition posets can be attached to Coxeter groups and play an important role in combinatorics and representation theory.

  Suppose that $(P,\leq_P)$ and $(Q,\leq_Q)$ are two posets. The \textit{direct sum} is the partial order $\leq_{P\times Q}$ on the set $P\times Q$ such that $(a,b)\leq_{P\times Q} (c,d)$ if and only if $a\leq_P c$ and $b\leq_Q d$. If $P$ and $Q$ are graded posets with degree maps $\operatorname{deg}_P$ and $\operatorname{deg}_Q$, then the direct product $(P\times Q,\leq_{P\times Q})$ is graded with collated degree map \linebreak$\operatorname{deg}_{P\times Q}(a,b)=\operatorname{deg}_P(a)+\operatorname{deg}_Q(b)$ for all $a\in P,~b\in Q$. The next theorem is a product theorem for symmetric chain orders. The main idea is due to de Bruijn, van Ebbenhorst Tengbergen and Kruyswijk \cite{BEK} and formal proofs are due to Aigner \cite{Ai}, Alekseev \cite{Al} and Griggs \cite{Gr}.

  \begin{theorem}
  If the graded posets $(P,\leq_P)$ and $(Q,\leq_Q)$ are both symmetric chain orders, then the direct product $(P\times Q,\leq_{P\times Q})$ is a symmetric chain order as well.
  \end{theorem}

  \begin{ex}
  \label{ex:chainprod}
  For a natural number $k$ the poset $Ch(k)=(0\leq 1\leq \ldots\leq k)$ is called the \textit{chain poset} of length $k+1$. It becomes a graded poset when we define $\operatorname{deg}(a)=a$ for all $0\leq a\leq k$. By construction the chain poset $Ch(k)$ is symmetric chain order. For natural numbers $k_1, k_2,\ldots,k_r$ let $Ch(k_1,k_2,\ldots,k_r)$ be the set of all sequences $(a_1,a_2,\ldots,a_r)\in \mathbb{N}^r$ such that $0\leq a_i\leq k_i$ for all $1\leq i\leq r$. We order the set by the dominance order, i.\,e. we say $(a_1,a_2,\ldots a_r)\leq (b_1,b_2,\ldots,b_r)$ if and only if $a_i\leq b_i$ for all $1\leq i\leq r$. Clearly, we have $Ch(k_1,k_2,\ldots,k_r)= Ch(k_1)\times Ch(k_2)\times\ldots\times Ch(k_r)$ and the product theorem implies that the poset $Ch(k_1,k_2,\ldots,k_r)$ is a symmetric chain order and hence Sperner. The poset is known as the \textit{chain product}.
  \end{ex}

  \subsection{Quiver representations}\label{subsec:quivrep}

  A \textit{quiver} $Q=(Q_0,Q_1)$ is a finite directed graph with vertex set $Q_0$ and arrow set $Q_1$. A vertex $i\in Q_0$ is called a \textit{source} if there does not exist an arrow $\alpha\in Q_1$ that ends in $i$. Similarly, a vertex $i\in Q_0$ is called a \textit{sink} if there does not exist an arrow $\alpha\in Q_1$ that starts in $i$.

  We fix a field $k$. A \textit{representation} of $Q$ is a collection $V=((V_i)_{i\in Q_0},(V_{\alpha})_{\alpha\in Q_1})$ consisting of a finite-dimensional $k$-vector space $V_i$ for every vertex $i\in Q_0$ and a $k$-linear map $V_{\alpha}\colon V_i\to V_j$ for every arrow $\alpha\colon i\to j$ in $Q_1$. We denote by $\underline{\operatorname{dim}} V=(\operatorname{dim}_{k} V_i)_{i\in Q_0}$ the \textit{dimension vector} of $V$. The \textit{support} of $V$ is defined as the set $\operatorname{supp}(V)=\{i\in Q_0\colon V_i\neq 0\}$. Furthermore, the sum $\operatorname{dim}_k(V)=\dim_{i\in Q_0}\operatorname{dim}_k(V_i)$ is called the \textit{dimension} of $V$.

  A \textit{subrepresentation} $U$ of $V$ is representation of $Q$ such that $U_i\subseteq V_i$ is a $k$-vector subspace for every vertex $i\in Q_i$ and $U_{\alpha}(x)=V_{\alpha}(x)$ for every arrow $\alpha\colon i\to j$ in $Q_1$ and every element $x\in U_i$. In particular, we have $U_{\alpha}(U_i)\subseteq U_j$ for every arrow $\alpha$. Given a subrepresentation $U \subseteq V$, we can define a quotient representation $V/U$ by vector spaces $(V/U)_i=V_i/U_i$ for all vertices $i\in Q_0$ and induced canonical maps $(V/U)_{\alpha}\colon V_i/U_i\to V_j/U_j$ for all arrows $\alpha\colon i\to j$. A representation $V$ is called \textit{simple} if it does not admit a non-zero proper subrepresentation $0 \subsetneqq U \subsetneqq V$. Suppose that $V,W$ are two representations of the same quiver $Q$. A \textit{morphism} $\phi\colon V\to W$ is a collection of $k$-linear maps $\phi_i\colon V_i\to W_i$ for all vertices $i\in Q_0$ such that $W_{\alpha}\circ\phi_i=\phi_j\circ V_{\alpha}$ for all arrows $\alpha\colon i\to j$ in $Q_1$. The morphism with $\phi_i = 0$ for all $i\in Q_0$ is called the \textit{zero morphism}. A morphism $\phi=(\phi_i)_{i\in Q_0}$ is called a \textit{monomorphism} if every linear map $\phi_i$ is injective. If $U$ is a subrepresentation of $V$, then the family of canonical inclusions $U_i \hookrightarrow V_i$ provides a basic example of a monomorphism $\phi:U\to V$.
  Dually, a morphism $\phi$ is called an \textit{epimorphism} if every linear map $\phi_i$ is surjective. A morphism $\phi\colon V\to W$ is called an \textit{isomorphism} if it is both a monomorphism and an epimorphism. In this case we say that $V$ and $W$ are \textit{isomorphic} and we write $V\cong W$.

  The representation $V$ with $V_i=0$ for all $i\in Q_0$ is called the \textit{zero representation}, where necessarily $V_{\alpha}=0$ for all $\alpha\in Q_1$. Suppose that $V,W$ are two representations of the same quiver $Q$. The \textit{direct sum} $V\oplus W$ is the representation with $(V\oplus W)_i=V_i\oplus W_i$ for all vertices $i\in Q_0$ and \[(V\oplus W)_{\alpha}=\left(\begin{smallmatrix}V_{\alpha}&0\\0&W_{\alpha}\end{smallmatrix}\right)\colon V_i\oplus W_i\to V_j\oplus W_j\] for all arrows $\alpha\colon i\to j$. A representation is called \textit{decomposable} if it is isomorphic to a direct sum $V\oplus W$ with $V,W\neq 0$. It is called \textit{indecomposable} otherwise. Note that every simple representation is indecomposable but the reverse statement does not hold in general. A quiver is called \textit{representation finite} if there are only finitely many indecomposable representations up to isomorphism. It is called \textit{representation infinite} otherwise.

  Let a \textit{path of length $n$} be an undirected graph as in Figure~\ref{fig:ungraph}.
  \begin{figure}
  	\centering
  	\begin{tikzpicture}[every node/.style={scale=0.775}]
      	\node (1) at (1,0) {1};
          \node (2) at (2,0) {2};
          \node (3) at (3,0) {3};
          \node (5) at (5,0) {$n$};
          \draw (1) to (2);
          \draw (2) to (3);
          \draw (3) -- (3.75,0);
          \draw[dotted] (3.8,0) -- (4.2,0);
          \draw (4.25,0) -- (5);
      \end{tikzpicture}
  	\caption{A path with $n$ vertices}\label{fig:ungraph}
  \end{figure}
  To unify the description of quivers which are representation finite, let us introduce \textit{star-shaped undirected graphs} as graphs with a central vertex $c$ from which $r$-many paths of varying lengths start. More formally, for integers $r\geq 0$ and $\ell_1, \ldots, \ell_r \geq\nolinebreak 1$ let $\operatorname{Star}(\ell_1, \ldots, \ell_r)$ be the graph with $n=1+\sum_{i=1}^r \ell_i$ many vertices and edges $\begin{tikzcd}c \ar[-]{r} &v_{i,1}\end{tikzcd}$\kern-0.5em, $\begin{tikzcd}v_{i,j} \ar[-]{r} &v_{i,j+1}\end{tikzcd}$ for $1\leq i \leq r$ and $1 \leq j \leq \ell_i-1$. Pictorially such a graph can be seen in Figure~\ref{fig:ungraphstar}.

  \begin{figure}
  	\centering
  	\begin{tikzpicture}[every node/.style={scale=0.775}]
      	\node (c) at (0:0) {$c$};
          \node (v11) 	at (140:1) {$v_{1,1}$};
          \node (v12) 	at (140:2) {$v_{1,2}$};
          \node (v13) 	at (140:3) {};
          \node (v1l1-1) 	at (140:4) {};
          \node (v1l1) 	at (140:5) {$v_{1,\ell_1}$};
          \node (v21) 	at (180:1) {$v_{2,1}$};
          \node (v22) 	at (180:2) {$v_{2,2}$};
          \node (v23) 	at (180:3) {};
          \node (v2l2-1) 	at (180:4) {};
      	\node (v2l2) 	at (180:5) {$v_{2,\ell_2}$};
          \node (v31) 	at (220:1) {$v_{3,1}$};
          \node (v32) 	at (220:2) {$v_{3,2}$};
          \node (v33) 	at (220:3) {};
          \node (v3l3-1) 	at (220:4) {};
      	\node (v3l3) 	at (220:5) {$v_{3,\ell_3}$};
          \node (vr-11) 	at (340:1) {$v_{r-1,1}$};
          \node (vr-12) 	at (340:2) {$v_{r-1,2}$};
          \node (vr-13) 	at (340:3) {};
          \node (vr-1lr-1-1)at (340:4) {};
      	\node (vr-1lr-1) 	at (340:5) {$v_{r-1,\ell_{r-1}}$};
          \node (vr1) 	at (20:1) {$v_{r,1}$};
          \node (vr2) 	at (20:2) {$v_{r,2}$};
          \node (vr3) 	at (20:3) {};
          \node (vrlr-1) 	at (20:4) {};
      	\node (vrlr) 	at (20:5) {$v_{r,\ell_r}$};

          \draw 	(c) -- (v11) (c) -- (v21) (c) -- (v31) (c) -- (vr-11) (c) -- (vr1)
          		(v11) -- (v12) (v12) -- (v13) (v1l1-1) -- (v1l1)
                  (v21) -- (v22) (v22) -- (v23) (v2l2-1) -- (v2l2)
                  (v31) -- (v32) (v32) -- (v33) (v3l3-1) -- (v3l3)
                  (vr-11) -- (vr-12) (vr-12) -- (vr-13) (vr-1lr-1-1) -- (vr-1lr-1)
                  (vr1) -- (vr2) (vr2) -- (vr3) (vrlr-1) -- (vrlr);
          \draw[dotted] (v13) -- (v1l1-1) (v23) -- (v2l2-1) (v33) -- (v3l3-1) (vr-13) -- (vr-1lr-1-1) (vr3) -- (vrlr-1);
          \draw[thick,dotted] ([shift=(240:2.5cm)]0,0) arc (240:320:2.5cm);
      \end{tikzpicture}
  	\caption{Star-shaped undirected graph}\label{fig:ungraphstar}
  \end{figure}
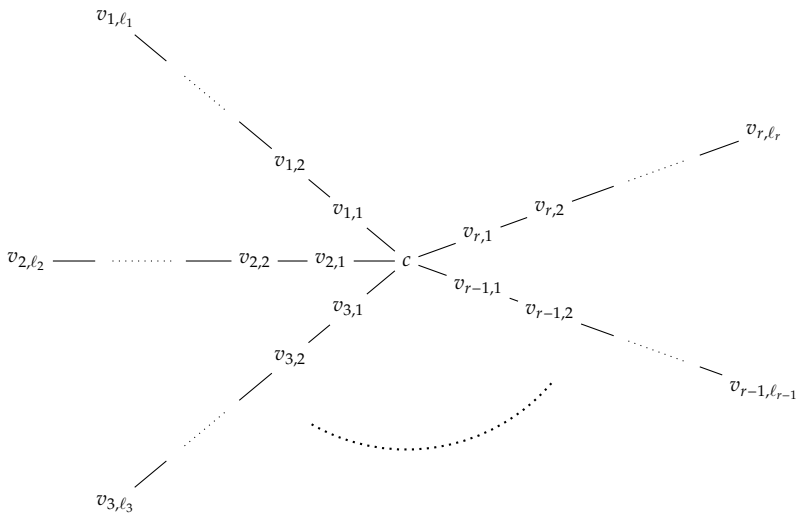

  A star-shaped undirected graph with $n$ vertices is said to be a \textit{Dynkin diagram of type $A_n$} if \linebreak$r=1, \ell_1 = n-1$, it is \textit{Dynkin diagram of type $D_n$} if $r=3,\ell_1=\ell_2=1,\ell_3=n-3$ and \textit{Dynkin diagram of type $E_n$} if $r=3, \ell_1=1, \ell_2=2$ and $\ell_3=n-4$ for $n \in \{6,7,8\}$.

  Gabriel \cite{G} then classifies representation finite quivers as follows:

  \begin{theorem}[Gabriel]\label{thm:gabriel}
  A (non-empty) connected quiver with $n$ vertices is representation finite if and only if its underlying undirected graph is a Dynkin diagram of type $A_n$, $D_n$ or $E_n$. In this case, the map $V\mapsto \underline{\operatorname{dim}} V$ induces a bijection between the isomorphism classes of indecomposable representations and the positive roots in the corresponding root system.
  \end{theorem}

  Especially, representation finiteness does only depend on the underlying diagram but not on the orientation. We say that quivers as in Theorem~\ref{thm:gabriel} are of type $A_n$, $D_n$ and $E_n$ respectively and call them \textit{Dynkin} if we do not wish to distinguish between these three families.
  For those readers not familiar with representation theory, let us consider one basic example to clarify the notions above.

  \begin{ex}\label{ex:a2_1}
  Let $Q$ be the quiver $1\stackrel{\alpha}{\longrightarrow} 2$ of type $A_2$.
  One representation is then given by $V_1 = k, V_2 = 0$ and the zero map; denote this representation by $S_1=(k\to 0)$. Since its only proper subrepresentation is the zero representation, we clearly see that it is simple. Similarly, $S_2=(0\to k)$ is a simple and thus also an indecomposable representation.
  Gabriel's theorem asserts the existence of a third indecomposable representation with dimension vector $(1,1)$, namely the representation $P_1 = (k \stackrel{\id}{\longrightarrow} k)$. It is an easy observation that the zero morphism is the only morphism from $S_1$ to $P_1$, i.\,e. the left diagram in Figure~\ref{fig:S1S2P1} commutes if and only if $\phi_1 = 0 = \phi_2$.
  \begin{figure}[H]
  	\centering
  	\begin{tikzpicture}
      	\node (P1) at (-0.75,0) {$P_1=$};
          \node (S1) at (-0.75,1.5) {$S_1=$};
      	\node (1) at (0,0) {$(k$};
          \node (2) at (2,0) {$k)$};
          \node (3) at (0,1.5) {$(k$};
          \node (4) at (2,1.5) {0)};
          \draw[-angle 90,relative] (1) to node[below] {$\id$}(2);
          \draw[-angle 90,relative] (3) to node[above] {0}(4);
          \draw[-angle 90,relative] (3) to node[left] {$\phi_1$} (1);
          \draw[-angle 90,relative] (4) to node[right] {$\phi_2$} (2);
      \end{tikzpicture}
      \hspace{2cm}
  	\begin{tikzpicture}
      	\node (P1) at (-0.75,0) {$P_1=$};
          \node (S2) at (-0.75,1.5) {$S_2=$};
      	\node (1) at (0,0) {$(k$};
          \node (2) at (2,0) {$k)$};
          \node (3) at (0,1.5) {$(0$};
          \node (4) at (2,1.5) {$k)$};
          \draw[-angle 90,relative] (1) to node[below] {$\id$}(2);
          \draw[-angle 90,relative] (3) to node[above] {0}(4);
          \draw[-angle 90,relative] (3) to node[left] {$\psi_1$} (1);
          \draw[-angle 90,relative] (4) to node[right] {$\psi_2$} (2);
      \end{tikzpicture}
  	\caption{Morphisms between indecomposable representations of a quiver of type $A_2$}\label{fig:S1S2P1}
  \end{figure}
  \noindent
  On the other hand, the choice $\psi_1=0$ and $\psi_2 =\id$ makes the right diagram of Figure~\ref{fig:S1S2P1} commutative, hence we obtain a nonzero morphism from $S_2$ to $P_1$. Since the identity map is injective, the morphism $(\psi_1,\psi_2)$ is even a monomorphism of representations.
  By Gabriel's theorem, a general representation has the form $V=S_1^{a}\oplus P_1^{b}\oplus S_2^{c}$ for some integers $a,b,c\geq 0$, so that $V_1=k^{a}\oplus k^b$, $V_2=k^b\oplus k^c$, and $V_{\alpha}=\left(\begin{smallmatrix}0&\id\\0&0\end{smallmatrix}\right)$ in block form.
  \end{ex}

  In this article, we wish to study maximum antichains in posets attached to various quivers representations. On a related note, Ringel \cite{R} studies maximal antichains in the product ordering on the set of dimension vectors of indecomposable representations of a Dynkin quiver.

  \section{Maximum antichains in monomorphism posets of indecomposable representations for type \texorpdfstring{$A$}{A} quivers}

  \subsection{Poset properties}
  Let $n\geq 1$ be a natural number and $Q$ be a Dynkin qiver of type $A_n$.
  For natural numbers $1\leq a\leq b\leq n$ use the shorthand notation $[a,b]$ for the representation with vector spaces $V_i = k$ for $a\leq i \leq b$ and $V_i=0$ elsewhere, and linear maps $V_\alpha = (V_i \to V_j) =\id_k$ for arrows $\alpha\colon i\to j$ with $a\leq i,j \leq b$ and $V_\alpha=0$ for all others. In the case where $a=b$ we simply write $[a]$ instead of $[a,a]$.

  \begin{ex}\label{ex:a2_2}
  Consider again the quiver $Q$ from Example~\ref{ex:a2_1}. Then $S_1=[1]$, $S_2=[2]$ and $P_1 = [1,2]$.
  \end{ex}

  As a consequence of Gabriel's theorem, every indecomposable representation of $Q$ is isomorphic to $[a,b]$ for appropriate choices of $1\leq a\leq b\leq n$.

  \begin{ex}\label{ex:zick_A3}
      Let $Q$ be the quiver $1 \leftarrow 2 \to 3$. The set of isomorphism classes of indecomposable representations contains 6 elements: $[1],~[2],~[3],~[1,2],~[2,3],~[1,3]$. These representations are visualized in Figure~\ref{fig:zick_A3}. To simplify the notation, whenever we draw an arrow between one-dimensional vector spaces, we assume that the associated map is the identity.
      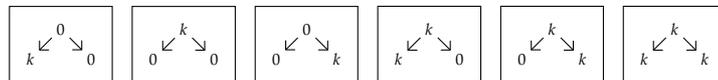
\begin{figure}[H]
              \centering
              \begin{tikzpicture}
                  \node (v1) {
                      \begin{tikzpicture}[scale=\scalingFactorSmall, every node/.style={scale=0.6}, framed]
                          \begin{scope}
                              \node[inner sep=.1cm] (1) at  (0,0) {$k$};
                              \node[inner sep=.1cm] (2) at  (1,1) {0};
                              \node[inner sep=.1cm] (3) at  (2,0) {0};
                              \draw[-angle 90,relative] (2) to (1);
                              \draw[-angle 90,relative] (2) to (3);
                          \end{scope}
                      \end{tikzpicture}};
                  \node[right=0cm of v1] (v2) {
                      \begin{tikzpicture}[scale=\scalingFactorSmall, every node/.style={scale=0.6}, framed]
                          \begin{scope}
                              \node[inner sep=.1cm] (1) at  (0,0) {0};
                              \node[inner sep=.1cm] (2) at  (1,1) {$k$};
                              \node[inner sep=.1cm] (3) at  (2,0) {0};
                              \draw[-angle 90,relative] (2) to (1);
                              \draw[-angle 90,relative] (2) to (3);
                          \end{scope}
                      \end{tikzpicture}};
                  \node[right=0cm of v2] (v3) {
                      \begin{tikzpicture}[scale=\scalingFactorSmall, every node/.style={scale=0.6}, framed]
                          \begin{scope}
                              \node[inner sep=.1cm] (1) at  (0,0) {0};
                              \node[inner sep=.1cm] (2) at  (1,1) {0};
                              \node[inner sep=.1cm] (3) at  (2,0) {$k$};
                              \draw[-angle 90,relative] (2) to (1);
                              \draw[-angle 90,relative] (2) to (3);
                          \end{scope}
                      \end{tikzpicture}};
                  \node [right=0cm of v3] (v1-2) {
                      \begin{tikzpicture}[scale=\scalingFactorSmall, every node/.style={scale=0.6}, framed]
                          \begin{scope}
                              \node[inner sep=.1cm] (1) at  (0,0) {$k$};
                              \node[inner sep=.1cm] (2) at  (1,1) {$k$};
                              \node[inner sep=.1cm] (3) at  (2,0) {0};
                              \draw[-angle 90,relative] (2) to (1);
                              \draw[-angle 90,relative] (2) to (3);
                          \end{scope}
                      \end{tikzpicture}};
                  \node [right=0cm of v1-2] (v2-3) {
                      \begin{tikzpicture}[scale=\scalingFactorSmall, every node/.style={scale=0.6}, framed]
                          \begin{scope}
                              \node[inner sep=.1cm] (1) at  (0,0) {0};
                              \node[inner sep=.1cm] (2) at  (1,1) {$k$};
                              \node[inner sep=.1cm] (3) at  (2,0) {$k$};
                              \draw[-angle 90,relative] (2) to (1);
                              \draw[-angle 90,relative] (2) to (3);
                          \end{scope}
                      \end{tikzpicture}};
                  \node [right=0cm of v2-3] (v1-3) {
                      \begin{tikzpicture}[scale=\scalingFactorSmall, every node/.style={scale=0.6}, framed]
                          \begin{scope}
                              \node[inner sep=.1cm] (1) at  (0,0) {$k$};
                              \node[inner sep=.1cm] (2) at  (1,1) {$k$};
                              \node[inner sep=.1cm] (3) at  (2,0) {$k$};
                              \draw[-angle 90,relative] (2) to (1);
                              \draw[-angle 90,relative] (2) to (3);
                          \end{scope}
                      \end{tikzpicture}};
              \end{tikzpicture}
          \caption{Indecomposable representations of alternating $A_3$}\label{fig:zick_A3}
      \end{figure}
  \end{ex}

  For the rest of this section, let $(\mathcal{P}_Q,\leq)$ be the poset with set $\mathcal{P}_Q=\big\{[a,b] \colon 1\leq a \leq b \leq n \big\}$ and $[a,b] \leq [a',b']$ whenever there exists a monomorphism $[a,b] \hookrightarrow [a',b']$. Equivalently, we have $[a,b] \leq [a',b']$ if and only if $[a,b] \subseteq [a',b']$.

  \begin{ex}\label{ex:a3}
  Consider again the quiver as in Example~\ref{ex:zick_A3}. Then $[1] \leq [1,2]$ and $[1] \leq [1,3]$, but \mbox{$[1,2] \not \leq [1,3]$}. The Hasse diagram of this poset is shown in Figure~\ref{fig:hasse_zick_a3}.
  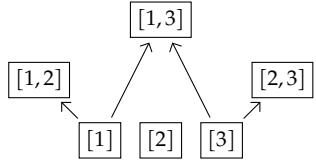
\begin{figure}
  	\centering
      \begin{tikzpicture}[scale=0.8, every node/.style={scale=0.8}]
      	\node[draw] (1) at (0,0) {$[1]$};
          \node[draw] (2) at (1,0) {$[2]$};
          \node[draw] (3) at (2,0) {$[3]$};
          \node[draw] (1-2) at (-1,1) {$[1,2]$};
          \node[draw] (2-3) at (3,1) {$[2,3]$};
          \node[draw] (1-3) at (1,2) {$[1,3]$};
          \draw[-angle 90,relative,shorten >=2pt,shorten <=2pt] (1) to (1-2);
          \draw[-angle 90,relative,shorten >=2pt,shorten <=2pt] (1) to (1-3);
          \draw[-angle 90,relative,shorten >=2pt,shorten <=2pt] (3) to (2-3);
          \draw[-angle 90,relative,shorten >=2pt,shorten <=2pt] (3) to (1-3);
      \end{tikzpicture}
    	\caption{Hasse diagram of indecomposable representations of alternating $A_3$}\label{fig:hasse_zick_a3}
  \end{figure}
  \end{ex}

  \begin{rem}
  The poset $(\mathcal{P}_Q,\leq)$ for $Q$ a Dynkin quiver of type $A_n$ may be graded as in Example~\ref{ex:a3}, but this is not generally the case as we observe later on in Example~\ref{ex:A6_-1-1111}. What is more, for $n>1$ these posets are never bounded and thus never Sperner as the simple representations supported on a source yield isolated vertices in the poset.
  \end{rem}

  \subsection{Linear orientation}\label{subsec:linear}
  In this and the following subsections, we will restrict to particular orientations of the path from Figure~\ref{fig:ungraph}. For now, we consider the linear orientation with unique sink at 1 and unique source at $n$, see Figure~\ref{fig:lin}.
  \begin{figure}[H]
  	\centering
  	\begin{tikzpicture}
      	\node (1) at (1,0) {1};
          \node (2) at (2,0) {2};
          \node (3) at (3,0) {3};
          \node (5) at (5,0) {$n$};
          \draw[-angle 90,relative] (2) to (1);
          \draw[-angle 90,relative] (3) to (2);
          \draw[-angle 90,relative] (3.75,0) -- (3);
          \draw[dotted] (3.8,0) -- (4.2,0);
          \draw[-angle 90,relative] (5) -- (4.25,0);
      \end{tikzpicture}
  	\caption{Linear orientation of type $A_n$}\label{fig:lin}
  \end{figure}

  Then $[a,b]\leq [a',b']$ if and only if $a=a'$ and $b\leq b'$. Hence the poset $(\mathcal{P}_Q,\leq)$ decomposes into $n$ disjoint chains $C_i = \Big([i] \leq [i,i+1] \leq [i,i+2] \leq \ldots \leq [i,n]\Big)$ for all $1\leq i \leq n$ and we immediately obtain the following result.

  \begin{prop}
  	A maximum antichain of $(\mathcal{P}_Q,\leq)$ for linearly oriented $A_n$ consists of exactly $n$ elements.
  \end{prop}

  \subsection{Simple zigzag}\label{subseq:zigzag}

  Let $1\leq s\leq n$ and consider the orientation of the path of length $n$ with a unique source at position $s$. For $s=1$ or $s=n$ the case degenerates to the linear orientation (up to reordering the vertices) of Subsection~\ref{subsec:linear}. To simplify the notation, denote $\ell=s-1$ and $r=n-s$ so that the quiver is of the form as shown in Figure~\ref{fig:zigzag}.
  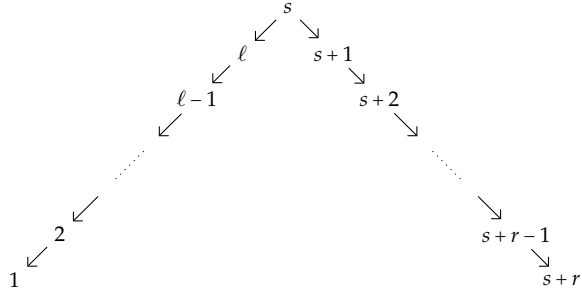
\begin{figure}[H]
  	\centering
  	\begin{tikzpicture}[scale=0.6, every node/.style={scale=0.8}]
      	\node (1) at (1,0) {1};
          \node (2) at (2,1) {2};
          \node (3) at (3,2) {};
          \node (l-2) at (4,3) {};
          \node (l-1) at (5,4) {$\ell-1$};
          \node (l) at (6,5) {$\ell$};
          \node (s) at (7,6) {$s$};
          \node (s+1) at (8,5) {$s+1$};
          \node (s+2) at (9,4) {$s+2$};
          \node (s+3) at (10,3) {};
          \node (s+r-2) at (11,2) {};
          \node (s+r-1) at (12,1) {$s+r-1$};
          \node (s+r) at (13,0) {$s+r$};
          \draw[-angle 90,relative] (2) to (1);
          \draw[-angle 90,relative] (3) to (2);
          \draw[dotted] (l-2) -- (3);
          \draw[-angle 90,relative] (l-1) to (l-2);
  		\draw[-angle 90,relative] (l) to (l-1);
          \draw[-angle 90,relative] (s) to (l);
          \draw[-angle 90,relative] (s) to (s+1);
  		\draw[-angle 90,relative] (s+1) to (s+2);
          \draw[-angle 90,relative] (s+2) to (s+3);
          \draw[dotted] (s+3) -- (s+r-2);
          \draw[-angle 90,relative] (s+r-2) to (s+r-1);
  		\draw[-angle 90,relative] (s+r-1) to (s+r);
      \end{tikzpicture}
  	\caption{Simple zigzag orientation of type $A_n$}\label{fig:zigzag}
  \end{figure}

  \begin{theorem}\label{thm:simplezigzag}
  	The set $\mathcal{F}=\Big\{ [a,b] \colon a \leq s \leq b \Big\}$ is a maximum antichain of size $(\ell+1)(r+1)$ in $(\mathcal{P}_Q,\leq)$.
  \end{theorem}
  \begin{proof}
  Let $[a,b]$ and $[a',b']$ be distinct elements from the set $\mathcal{F}$ and suppose there exists a monomorphism $\phi\colon [a,b] \hookrightarrow [a',b']$. This implies $a' \leq a$ and $b\leq b'$. Since these elements are distinct, $a'<a$ or $b<b'$. Without loss of generality, we may constrict to $a' < a$. Then the linear map corresponding to $\alpha\colon a\to (a-1)$ is zero in $[a,b]$ whereas in $[a',b']$ it is the identity. But $\id_k \circ \phi_{a} \neq \phi_{a-1} \circ 0$ yields a contradiction. Hence $\mathcal{F}$ is an antichain.

  To show that $\mathcal{F}$ is a maximum antichain, we describe a particular chain decomposition of $(\mathcal{P}_Q,\leq)$ and observe that $\mathcal{F}$ contains precisely one element of each chain.

  For $1\leq i \leq \ell$ and $1\leq j \leq r$ let
  \begin{align*}
  	L_i &= \Big( [i] \leq [i,i+1] \leq \ldots \leq [i,s] \Big),\\
      R_j &= \Big( [j] \leq [j-1,j] \leq \ldots \leq [s,j] \Big).
  \end{align*}
  We then obtain a chain decomposition of the entire poset $\mathcal{P}$ as follows:
  \begin{equation}\label{eq:chaindecomp}
  	\mathcal{P} = \left(\bigsqcup_{1\leq i \leq \ell} L_i \right) \sqcup \left(\bigsqcup_{1\leq j \leq r} R_j \right) \sqcup \left( \bigsqcup_{a < s < b} \Big\{[a,b]\Big\} \right) \sqcup \Big\{ [s] \Big\}.
  \end{equation}
  Every element of $\mathcal{F}$ lies in exactly one chain in~\eqref{eq:chaindecomp} as the maximal element. The number of elements in an antichain is thus bounded above by the number of chains in the decomposition in~\eqref{eq:chaindecomp}, which coincides with the cardinality of $\mathcal{F}$ by construction. This implies that the decomposition in~\eqref{eq:chaindecomp} is a Dilworth decomposition and that $\mathcal{F}$ is a maximum antichain.
  \end{proof}

  \begin{ex}\label{ex:A6_-1-1111}
      Let $n=6$ and $s=3$, hence $\ell = 2$ and $r=3$. The Hasse diagram drawn horizontally and $[a,b] \leq [a',b']$ indicated by arrows $[a,b] \to [a',b']$ is shown in Figure~\ref{fig:A6_-1-1111}. The elements of the set $\mathcal{F}$ are highlighted.
      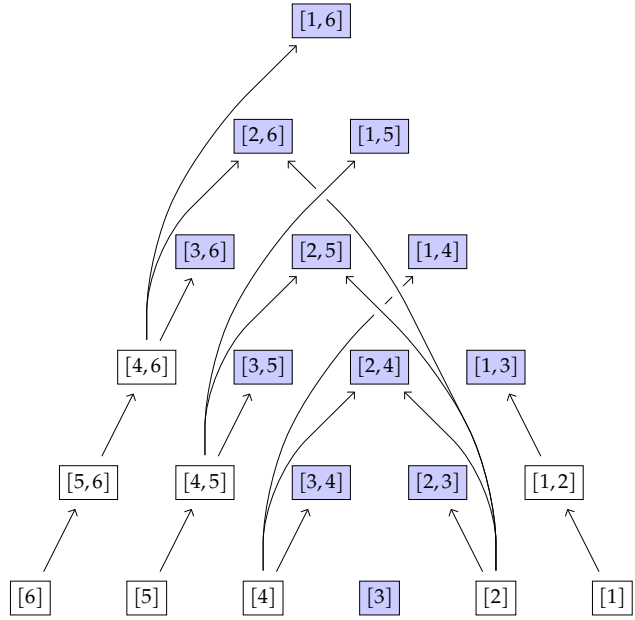
\begin{figure}
              \centering
              \resizebox{.5\linewidth}{!}{
              \begin{tikzpicture}

                  \node[draw] (v6) at (0,0)    { $[6]$ };
                  \node[draw] (v5) at (2,0)    { $[5]$ };
                  \node[draw] (v4) at (4,0)    { $[4]$ };
                  \node[draw,fill=blue!20] (v3) at (6,0)    { $[3]$ };
                  \node[draw] (v2) at (8,0)    { $[2]$ };
                  \node[draw] (v1) at (10,0)   { $[1]$ };
                  \node[draw] (v5-6) at (1,2)  { $[5,6]$ };
                  \node[draw] (v4-5) at (3,2)  { $[4,5]$ };
                  \node[draw,fill=blue!20] (v3-4) at (5,2)  { $[3,4]$ };
                  \node[draw,fill=blue!20] (v2-3) at (7,2)  { $[2,3]$ };
                  \node[draw] (v1-2) at (9,2)  { $[1,2]$ };
                  \node[draw] (v4-6) at (2,4)  { $[4,6]$ };
                  \node[draw,fill=blue!20] (v3-5) at (4,4)  { $[3,5]$ };
                  \node[draw,fill=blue!20] (v2-4) at (6,4)  { $[2,4]$ };
                  \node[draw,fill=blue!20] (v1-3) at (8,4)  { $[1,3]$ };
                  \node[draw,fill=blue!20] (v3-6) at (3,6)  { $[3,6]$ };
                  \node[draw,fill=blue!20] (v2-5) at (5,6)  { $[2,5]$ };
                  \node[draw,fill=blue!20] (v1-4) at (7,6)  { $[1,4]$ };
                  \node[draw,fill=blue!20] (v2-6) at (4,8) { $[2,6]$ };
                  \node[draw,fill=blue!20] (v1-5) at (6,8) { $[1,5]$ };
                  \node[draw,fill=blue!20] (v1-6) at (5,10) { $[1,6]$ };

                  \draw[-angle 90,relative,shorten >=5pt,shorten <=5pt] (v6) to (v5-6);
                  \draw[-angle 90,relative,shorten >=5pt,shorten <=5pt] (v5) to (v4-5);
                  \draw[-angle 90,relative,shorten >=5pt,shorten <=5pt] (v4) to (v3-4);
                  \draw[-angle 90,relative,rounded corners=30pt,shorten >=5pt,shorten <=5pt] (v4) -- (4,2) -- (v2-4);
                  \draw[-angle 90,relative,rounded corners=30pt,shorten >=5pt,shorten <=5pt] (v4) -- (4,2) -- (5,4) -- (v1-4);
                  \draw[fill=white,draw=white] (6,5) circle [radius=0.125];
                  \draw[fill=white,draw=white] (6.325,5.3) circle [radius=0.125];
                  \draw[-angle 90,relative,shorten >=5pt,shorten <=5pt] (v2) to (v2-3);
                  \draw[-angle 90,relative,rounded corners=30pt,shorten >=5pt,shorten <=5pt] (v2) -- (8,2) -- (v2-4);
                  \draw[-angle 90,relative,rounded corners=30pt,shorten >=5pt,shorten <=5pt] (v2) -- (8,2) -- (7,4) -- (v2-5);
                  \draw[-angle 90,relative,rounded corners=30pt,shorten >=5pt,shorten <=5pt] (v2) -- (8,2) -- (7,4) -- (6,6) -- (v2-6);
                  \draw[fill=white,draw=white] (5,7) circle [radius=0.125];
                  \draw[-angle 90,relative,shorten >=5pt,shorten <=5pt] (v1) to (v1-2);
  				        \draw[-angle 90,relative,shorten >=5pt,shorten <=5pt] (v5-6) to (v4-6);
                  \draw[-angle 90,relative,shorten >=5pt,shorten <=5pt] (v4-5) to (v3-5);
                  \draw[-angle 90,relative,rounded corners=30pt,shorten >=5pt,shorten <=5pt] (v4-5) -- (3,4) -- (v2-5);
                  \draw[-angle 90,relative,rounded corners=30pt,shorten >=5pt,shorten <=5pt] (v4-5) -- (3,4) -- (4,6) -- (v1-5);
                  \draw[-angle 90,relative,shorten >=5pt,shorten <=5pt] (v1-2) to (v1-3);
                  \draw[-angle 90,relative,shorten >=5pt,shorten <=5pt] (v4-6) to (v3-6);
                  \draw[-angle 90,relative,rounded corners=30pt,shorten >=5pt,shorten <=5pt] (v4-6) -- (2,6) -- (v2-6);
                  \draw[-angle 90,relative,rounded corners=30pt,shorten >=5pt,shorten <=5pt] (v4-6) -- (2,6) -- (3,8) -- (v1-6);

              \end{tikzpicture}}
          \caption{Monomorphism poset of a quiver of type $A_6$}
          \label{fig:A6_-1-1111}
      \end{figure}
  \end{ex}

  \subsection{Alternating orientation}\label{subseq:alternating}

  For this subsection, let $m\geq 1$ be a natural number and $n=2m+1$. Then we consider the quiver $Q$ which arises from the alternating orientation of the path of length $n$ starting and ending with a sink, see Figure~\ref{fig:alt_5}.

  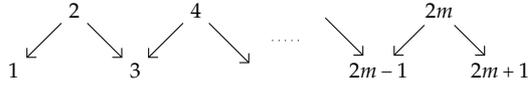
\begin{figure}
  	\centering
  	\begin{tikzpicture}[scale=0.8, every node/.style={scale=0.8}]
      	\node (1) at (1,0) {1};
          \node (2) at (2,1) {2};
          \node (3) at (3,0) {$3$};
          \node (4) at (4,1) {$4$};
          \node (5) at (5,0) {};
          \node (2s-2) at (6,1) {};
          \node (2s-1) at (7,0) {$2m-1$};
          \node (2s) at (8,1) {$2m$};
          \node (2s+1) at (9,0) {$2m+1$};
          \draw[-angle 90,relative] (2) to (1);
          \draw[-angle 90,relative] (2) to (3);
          \draw[-angle 90,relative] (4) to (3);
          \draw[-angle 90,relative] (4) to (5);
          \draw[-angle 90,relative] (2s-2) to (2s-1);
  		\draw[-angle 90,relative] (2s) to (2s-1);
          \draw[-angle 90,relative] (2s) to (2s+1);
          \draw[dotted] (5.25,0.5) -- (5.75,0.5);
      \end{tikzpicture}
  	\caption{Alternating orientation of path of length $n$ with $m$ sources}\label{fig:alt_5}
  \end{figure}

  \begin{theorem}\label{thm:alternating}
  Let $\mathcal{F} = \mathcal{F}_{\text{src}} \sqcup \mathcal{F}_{1-\text{src}} \sqcup \mathcal{F}_{\text{src}-n} \sqcup \mathcal{F}_{1-n}$ be the union of the following four sets
      \begin{align*}
      	\mathcal{F}_{\text{src}} &=\big\{ [a,b] \colon a\leq b \text{ sources}\big\}, & \mathcal{F}_{\text{src}-n}&= \big\{ [a,n] \colon a \text{ source}\big\},\\
          \mathcal{F}_{1-\text{src}} &= \big\{ [1,b] \colon b \text{ source}\big\},& \mathcal{F}_{1-n} &= \big\{ [1,n] \big\}.
      \end{align*}
      Then $\mathcal{F}$ is a maximum antichain of cardinality $\frac{1}{2}m(m+1) + 2m + 1$.
  \end{theorem}
  \begin{proof}
  	Using the same arguments as in the proof of Theorem~\ref{thm:simplezigzag}, we see that $\mathcal{F}$ is an antichain. And as before, we provide a chain decomposition of the entire poset to show that it is also maximum.

      Let $a$ and $b$ be sources. Then consider the chains in Table~\ref{tab:chains} depending on a choice of one or two sources.
      \begin{table}[H]
          \begin{center}
              \begin{tabular}{ lccc }
              \toprule\addlinespace[1ex]
              \multicolumn{2}{c}{Chain} \\ \cmidrule(r){1-2}\addlinespace[1ex]
              Name & Description & Condition & Cardinality \\ \midrule[1.5pt]\addlinespace[1ex]
              $C_{[a]}$ & $\Big( [a] \Big)$ & & 1\\ \addlinespace[1ex]
              $C_{[a,n]}$ & $\Big( [a,n] \Big)$ & $a\neq 2$ & 1 \\ \addlinespace[1ex]
              $C_{[1,a]}$ & $\Big( [a-1,a] \leq [a-3,a] \leq \ldots \leq [3,a] \leq [1,a] \Big)$ & & $a/2$\\ \addlinespace[1ex]
              $C_{[a,b]}$ & $\Big( [a+1,b-1] \leq [a,b-1] \leq [a,b] \Big)$ & $a < b$ & 3\\\addlinespace[1ex]
              \bottomrule
              \end{tabular}
          \end{center}
          \caption{Chains depending on a choice of one or two sources}\label{tab:chains}
  	\end{table}
  	\noindent
      Furthermore, we also consider two chains $C_{[1,n]}$ and $C_{[2,n]}$ not depending on a choice of sources, both of cardinality $m+1$:
      \begin{align*}
      	C_{[1,n]} &= \Big( [1] \leq [1,3] \leq [1,5] \leq \ldots \leq [1,n] \Big), \\
          C_{[2,n]} & = \Big( [n] \leq [n-2,n] \leq [n-4,n] \leq \ldots \leq [5,n] \leq [3,n] \leq [2,n] \Big).
      \end{align*}
      By construction, all of the chains above are pairwise disjoint. Altogether these chains exhaust all elements of the poset since
      \[
      	\frac{n(n+1)}{2} = \frac{(2m+1)(2m+2)}{2} = 2m^2+3m+1 = m\cdot 1 + (m-1) \cdot 1 + \frac{m(m+1)}{2} + \binom{m}{2} \cdot 3 + 2 \cdot (m+1).
      \]
      Thus they form a chain decomposition of $(\mathcal{P}_Q,\leq)$ and every element of $\mathcal{F}$ lies in exactly one chain as the maximal element.
  \end{proof}

  We can apply this construction to a particular subquiver of $Q$. Let $Q'$ be the full subquiver of $Q$ with the vertex $n$ deleted, i.\,e. the alternating orientation of the path of length $n-1$ starting with a sink and ending with a source. Then we notate $n'=n-1$ and observe that $Q'$ has indeed $n'=2m$ vertices and $s$ sources.

  \begin{cor}
  	Let $\mathcal{F}'_{\text{src}} = \mathcal{F}_{\text{src}}\setminus \big\{[2,n']\big\}$ and $\mathcal{F}'_{1-\text{src}} = \mathcal{F}_{1-\text{src}} \setminus \big\{ [1,n'] \big\}$. Then
      \begin{equation*}
      	\mathcal{F}' = \mathcal{F}'_{\text{src}} \sqcup \mathcal{F}'_{1-\text{src}} \sqcup \big\{ [1,n'-1] \big\} \sqcup \big\{ [2,n'-1] \big\} \sqcup \big\{ [3,n'] \big\}
      \end{equation*}\
      is a maximum antichain of $(\mathcal{P}_{Q'},\leq)$ of cardinality $\frac{1}{2}(m+1)(m+2)$.
  \end{cor}
  \begin{proof}
  	The Dilworth decomposition in the proof of Theorem~\ref{thm:alternating} degenerates to a Dilworth decomposition of $\mathcal{P}_{Q'}$ if one removes all those elements supported at vertex $n$.
  \end{proof}

  \begin{ex}

  	Let us consider the case for $m=3$, hence $n=7$ and $n'=6$. The Hasse diagrams of the posets $(\mathcal{P}_Q, \leq)$ and $(\mathcal{P}_{Q'}, \leq)$ are shown in Figure~\ref{fig:alt_Hasse_67}. Those nodes contained in $\mathcal{P}_Q$ but not in $\mathcal{P}_{Q'}$ are shaded gray above the downward diagonal. The elements in $\mathcal{F}$  are highlighted in blue below the downward diagonal and those of $\mathcal{F}'$ in red above the downward diagonal.

  	\begin{figure}
  		\centering
          \resizebox{.6\linewidth}{!}{
          \begin{tikzpicture}
          		\node[draw] (1) at (1,0) {$[1]$};
              \node[draw] (3) at (4,0) {$[3]$};
              \node[draw] (5) at (7,0) {$[5]$};
              \node[rectangle with diagonal fill, diagonal top color=black!20, diagonal bottom color=white, diagonal from left to right, draw] (7) at (10,0) {$[7]$};
              \node[rectangle with diagonal fill, diagonal top color=red!50, diagonal bottom color=blue!30, diagonal from left to right, draw] (1-2) at (0,2) {$[1,2]$};
              \node[draw] (1-3) at (2.2,2) {$[1,3]$};
              \node[draw] (3-4) at (3.3,2) {$[3,4]$};
              \node[draw] (2-3) at (4.4,2) {$[2,3]$};
              \node[draw] (3-5) at (5.5,2) {$[3,5]$};
              \node[draw] (5-6) at (6.6,2) {$[5,6]$};
              \node[draw] (4-5) at (7.7,2) {$[4,5]$};
              \node[rectangle with diagonal fill, diagonal top color=black!20, diagonal bottom color=white, diagonal from left to right, draw] (5-7) at (8.8,2) {$[5,7]$};
              \node[rectangle with diagonal fill, diagonal top color=black!20, diagonal bottom color=blue!30, diagonal from left to right, draw] (6-7) at (11,2) {$[6,7]$};
              \node[rectangle with diagonal fill, diagonal top color=red!50, diagonal bottom color=blue!30, diagonal from left to right, draw] (1-4) at (1,4) {$[1,4]$};
              \node[rectangle with diagonal fill, diagonal top color=red!50, diagonal bottom color=blue!30, diagonal from left to right, draw] (2-4) at (2.1,4) {$[2,4]$};
              \node[rectangle with diagonal fill, diagonal top color=red!50, diagonal bottom color=white, diagonal from left to right, draw] (1-5) at (3.85,4) {$[1,5]$};
              \node[rectangle with diagonal fill, diagonal top color=red!50, diagonal bottom color=white, diagonal from left to right, draw] (2-5) at (4.95,4) {$[2,5]$};
              \node[rectangle with diagonal fill, diagonal top color=red!50, diagonal bottom color=white, diagonal from left to right, draw] (3-6) at (6.05,4) {$[3,6]$};
              \node[rectangle with diagonal fill, diagonal top color=black!20, diagonal bottom color=white, diagonal from left to right, draw] (3-7) at (7.15,4) {$[3,7]$};
              \node[rectangle with diagonal fill, diagonal top color=red!50, diagonal bottom color=blue!30, diagonal from left to right, draw] (4-6) at (8.9,4) {$[4,6]$};
              \node[rectangle with diagonal fill, diagonal top color=black!20, diagonal bottom color=blue!30, diagonal from left to right, draw] (4-7) at (10,4) {$[4,7]$};
              \node[rectangle with diagonal fill, diagonal top color=white, diagonal bottom color=blue!30, diagonal from left to right, draw] (1-6) at (3,6) {$[1,6]$};
              \node[rectangle with diagonal fill, diagonal top color=black!20, diagonal bottom color=blue!30, diagonal from left to right, draw] (1-7) at (5.5,8) {$[1,7]$};
              \node[rectangle with diagonal fill, diagonal top color=white, diagonal bottom color=blue!30, diagonal from left to right, draw]  (2-6) at (5.5,6) {$[2,6]$};
              \node[rectangle with diagonal fill, diagonal top color=black!20, diagonal bottom color=blue!30, diagonal from left to right, draw] (2-7) at (8,6) {$[2,7]$};
              \node[rectangle with diagonal fill, diagonal top color=red!50, diagonal bottom color=blue!30, diagonal from left to right, draw] (2) at (2.5,-1) {$[2]$};
              \node[rectangle with diagonal fill, diagonal top color=red!50, diagonal bottom color=blue!30, diagonal from left to right, draw] (4) at (5.5,-1) {$[4]$};
              \node[rectangle with diagonal fill, diagonal top color=red!50, diagonal bottom color=blue!30, diagonal from left to right, draw] (6) at (8.5,-1) {$[6]$};

              \draw[-angle 90,relative, shorten >=5pt, shorten <=5pt] (1) to (1-2);
              \draw[-angle 90,relative, shorten >=5pt, shorten <=5pt] (1) to (1-3);
              \draw[-angle 90,relative, shorten >=5pt, shorten <=5pt] (3) to (1-3);
              \draw[-angle 90,relative, shorten >=5pt, shorten <=5pt] (3) to (3-4);
              \draw[-angle 90,relative, shorten >=5pt, shorten <=5pt] (3) to (2-3);
              \draw[-angle 90,relative,shorten >=5pt,shorten <=5pt] (3) to (3-5);
              \draw[-angle 90,relative,shorten >=5pt,shorten <=5pt] (5) to (3-5);
              \draw[-angle 90,relative,shorten >=5pt,shorten <=5pt] (5) to (5-6);
              \draw[-angle 90,relative,shorten >=5pt,shorten <=5pt] (5) to (4-5);
              \draw[-angle 90,relative,shorten >=5pt,shorten <=5pt] (5) to (5-7);
              \draw[-angle 90,relative,shorten >=5pt,shorten <=5pt] (7) to (5-7);
              \draw[-angle 90,relative,shorten >=5pt,shorten <=5pt] (7) to (6-7);
              \draw[-angle 90,relative,shorten >=5pt,shorten <=5pt] (1-3) to (1-4);
              \draw[-angle 90,relative,shorten >=5pt,shorten <=5pt] (1-3) to (1-5);
              \draw[fill=white,draw=white] (2.66,2.57) circle [radius=0.075];
              \draw[fill=white,draw=white] (2.85,2.775) circle [radius=0.075];
              \draw[fill=white,draw=white] (3.125,3.1) circle [radius=0.075];
              \draw[-angle 90,relative,shorten >=5pt,shorten <=5pt] (3-4) to (1-4);
              \draw[-angle 90,relative,shorten >=5pt,shorten <=5pt] (3-4) to (2-4);
              \draw[-angle 90,relative,shorten >=5pt,shorten <=5pt] (2-3) to (2-4);
              \draw[-angle 90,relative,shorten >=5pt,shorten <=5pt] (2-3) to (2-5);
              \draw[fill=white,draw=white] (4.675,3) circle [radius=0.075];
              \draw[-angle 90,relative,shorten >=5pt,shorten <=5pt] (3-5) to (1-5);
              \draw[-angle 90,relative,shorten >=5pt,shorten <=5pt] (3-5) to (2-5);
              \draw[-angle 90,relative,shorten >=5pt,shorten <=5pt] (3-5) to (3-6);
              \draw[-angle 90,relative,shorten >=5pt,shorten <=5pt] (3-5) to (3-7);
              \draw[fill=white,draw=white] (6.325,3) circle [radius=0.075];
              \draw[-angle 90,relative,shorten >=5pt,shorten <=5pt] (5-6) to (3-6);
              \draw[-angle 90,relative,shorten >=5pt,shorten <=5pt] (5-6) to (4-6);
              \draw[fill=white,draw=white] (7.9,3.1) circle [radius=0.075];
              \draw[-angle 90,relative,shorten >=5pt,shorten <=5pt] (4-5) to (4-6);
              \draw[fill=white,draw=white] (8.175,2.75) circle [radius=0.075];
              \draw[-angle 90,relative,shorten >=5pt,shorten <=5pt] (4-5) to (4-7);
              \draw[fill=white,draw=white] (8.35,2.575) circle [radius=0.075];
              \draw[-angle 90,relative,shorten >=5pt,shorten <=5pt] (5-7) to (3-7);
              \draw[-angle 90,relative,shorten >=5pt,shorten <=5pt] (5-7) to (4-7);
              \draw[-angle 90,relative,shorten >=5pt,shorten <=5pt] (1-5) to (1-6);
              \draw[-angle 90,relative,shorten >=5pt,shorten <=5pt] (1-5) to (1-7);
              \draw[fill=white,draw=white] (4.35,5.25) circle [radius=0.075];
              \draw[-angle 90,relative,shorten >=5pt,shorten <=2pt] (2-5.north) to (2-6);
              \draw[fill=white,draw=white] (5.125,4.8) circle [radius=0.075];
              \draw[-angle 90,relative,shorten >=5pt,shorten <=5pt] (2-5.north) to (2-7);
              \draw[fill=white,draw=white] (5.5,4.625) circle [radius=0.075];
              \draw[fill=white,draw=white] (5.875,4.8) circle [radius=0.075];
              \draw[fill=white,draw=white] (6.65,5.25) circle [radius=0.075];
              \draw[-angle 90,relative,shorten >=5pt,shorten <=5pt] (3-6.north) to (1-6);
              \draw[-angle 90,relative,shorten >=5pt,shorten <=2pt] (3-6.north) to (2-6);
              \draw[-angle 90,relative,shorten >=5pt,shorten <=5pt] (3-7) to (1-7);
              \draw[-angle 90,relative,shorten >=5pt,shorten <=5pt] (3-7) to (2-7);

          \end{tikzpicture}}
          \caption{Hasse diagrams of alternating orientations of $A_6$ and $A_7$}\label{fig:alt_Hasse_67}
      \end{figure}
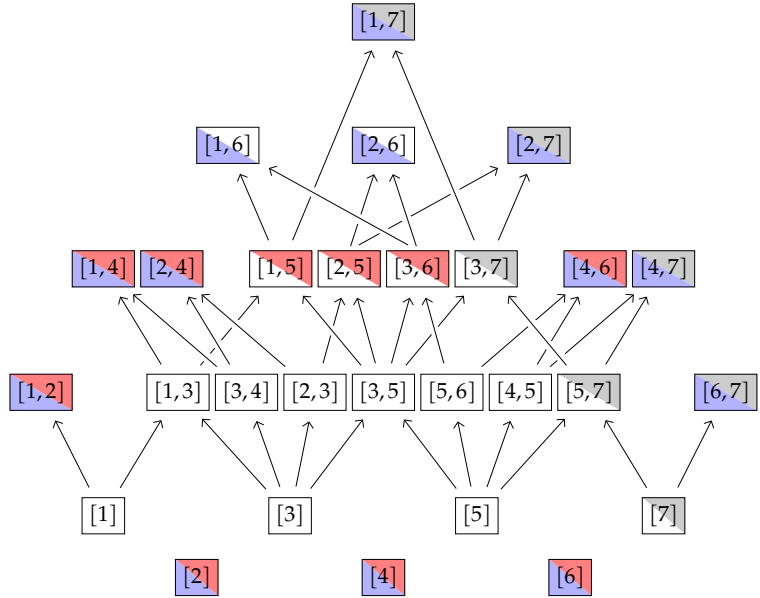
  \end{ex}

  \begin{rem}
  	For the only case where the alternating orientation of this section coincides with the simple zigzag of Subsection~\ref{subseq:zigzag}, the maximum antichains of Theorem~\ref{thm:simplezigzag} and Theorem~\ref{thm:alternating} are identical.
  \end{rem}

  \begin{rem}
  	The cases considered in this section should be extended to arbitrary orientations of quivers of type $A_n$, and even to the types $D_n$ and $E_{6,7,8}$. Computational experiments suggest that the combinatorics of general Dynkin cases are much more intricate than what we have previously seen.
  \end{rem}

  \section{Sperner theorems for subrepresentation posets in type \texorpdfstring{$A$}{A}}

  \subsection{A Sperner theorem for subrepresentations posets in type \texorpdfstring{$A_2$}{A2}}
  \label{sub:SpernerLinearSub}

  Let $V$ be a representation of a quiver $Q=(Q_0,Q_1)$ over a field $k$. We denote by $\mathcal{P}_V$ the set of all subrepresentations $U\subseteq V$. Assume that $U_1,U_2\in\mathcal{P}_V$. We say that $U_1\leq U_2$ if and only if $U_1$ is a subrepresentation of $U_2$. In this way, $(\mathcal{P}_V,\leq)$ becomes a partially ordered set.

  \begin{prop}
  \label{prop:grading}
  If every simple representation of $Q$ is $1$-dimensional, then the dimension function \linebreak$U\mapsto\operatorname{dim}_k(U)$ defines a grading of the subrepresentation poset $(\mathcal{P}_V,\leq)$.
  \end{prop}

  \begin{proof} The only minimal element in $\mathcal{P}_V$ is the zero representation which satisfies $\operatorname{dim}(0)=0$. Suppose that $U_2\in\mathcal{P}_V$ covers $U_1\in\mathcal{P}_V$, i.\,e. $U_1\leq U_2$ and there does not exist an element $U_3\in\mathcal{P}_V$ such that $U_1<U_3<U_2$. The third isomorphism theorem implies that the quotient representation $U_2/U_1$ is nonzero and that there does not exist a representation $0\subsetneqq U\subsetneqq U_2/U_1$. Thus $U_2/U_1$ is simple and the assumption implies $\operatorname{dim}_k(U_2)-\operatorname{dim}_k(U_1)=1$.
  \end{proof}

  Note that the condition of the previous proposition is satisfied if the quiver $Q$ does not contain oriented cycles which we assume from now on. Additionally, we assume that $k=\mathbb{F}_q$ is a finite field $\mathbb{F}_q$ with $q$ elements. For every natural number $n\in\mathbb{N}$ we define the \textit{Gaussian integer} as $[n]_q=(q^n-1)/(q-1)$. Note that $[n]_q=1+q+q^2+\ldots+q^{n-1}$ can be simplified to a polynomial in $q$ which specializes to $n$ when we plug in $q=1$. Moreover, $[n]_q$ is equal to the number of $1$-dimensional vector subspaces of $\mathbb{F}_q^n$. The polynomial $[n]_q!=[1]_q\cdot[2]_2\cdot\ldots\cdot[n]_q$ is called \textit{Gaussian factorial}. For a natural number $0\leq d\leq n$ the polynomial $\binom{n}{d}_q=[n]_q!/([d]_q!\cdot[n-d]_q!)$ is called \textit{Gaussian binomial coefficient}. Note that $\binom{n}{d}_q$ is equal to the number of $d$-dimensional vector subspaces of $\mathbb{F}_q^n$. Likewise the number of $(d+1)$-dimensional $k$-vector spaces $U$ such that $k^d\subseteq U\subseteq k^n$ is equal to $[n-d]_q$. Dually, the number of $(n-1)$-dimensional vector spaces $U$ such that $k^d\subseteq U\subseteq k^n$ is also equal to $[n-d]_q$.

  \begin{theorem}
  If $Q=(1\to 2)$ is the quiver of type $A_2$ and $V=P_1^a$ is a direct sum of copies of the indecomposable, projective representation $P_1$, then the subrepresentation poset $(\mathcal{P}_V,\leq)$ is Sperner.
  \end{theorem}

  \begin{proof}
  Note that the rank of the poset is equal to $\operatorname{dim}_k(V)=2a$. For brevity we write $\mathcal{P}_{i,V}$ instead of $(\mathcal{P}_{V})_{i}$ for $0\leq i\leq 2a$. A subrepresentation $X\subseteq V$ is given by two vector spaces $0\subseteq X_1\subseteq X_2\subseteq k^a$. We apply Stanley's Theorem \ref{thm:Stanley}. For every integer $i$ we define two maps
  \begin{align*}
    U_i\colon\mathbb{Q}\mathcal{P}_{i,V}&\to\mathcal{P}_{i+1,V}, &D_i\colon\mathbb{Q}\mathcal{P}_{i,V}&\to\mathcal{P}_{i-1,V},\\
    X&\mapsto \sum_{\substack{Y\in\mathcal{P}_{i+1,V}\\ Y\geq X}}Y,&X &\mapsto \sum_{\substack{Z\in\mathcal{P}_{i-1,V}\\ Z\leq X}} Z.
  \end{align*}
  for every $X\in\mathcal{P}_{i,V}$ and extend linearly. By construction $D_i$ is adjoint to $U_{i-1}$ for every $i$. We can generalize Stanley's commutation relation for the up and down operators from type $A_1$ to type $A_2$. More precisely, we claim
  \begin{align}
    \Big(D_{i+1}\circ U_i-U_{i-1}\circ D_i\Big)(X)=\Big([a-d_2]_q-[d_1]_q\Big)X
    \label{eqn:commutator}
  \end{align}
  for every $X\in\mathcal{P}_{i,V}$ with dimension vector $(d_1,d_2)$.

  For a proof of the claim, let $X\in\mathcal{P}_{i,V}$. We compute the coefficients of the expansion of \linebreak$(D_{i+1}\circ U_i-U_{i-1}\circ D_i)(X)$ in the standard basis $\mathcal{P}_{i,V}$ of $\mathbb{Q}\mathcal{P}_{i,V}$. We distinguish the following cases. First, suppose that $X'\in\mathcal{P}_{i,V}$ is a representation such that $X_1'\subseteq X_1$ and $X_2\subseteq X_2'$ are both subspaces of codimension $1$. In this case, $(X_1\subseteq X_2')$ and $(X_1'\subseteq X_2)$ are both subrepresentations of $V$ so that the coefficient corresponding to $X'$ in the expansion of $(D_{i+1}\circ U_i-U_{i-1}\circ D_i)(X)$ is equal to $1-1=0$. Second, suppose that $X'\in\mathcal{P}_{i,V}$ is a representation such that $X_1\subseteq X_1'$ and $X_2'\subseteq X_2$ are both subspaces of codimension $1$. With the same arguments as before the coefficient corresponding to $X'$ in the expansion of $(D_{i+1}\circ U_i-U_{i-1}\circ D_i)(X)$ is equal to $1-1=0$. Third, suppose that $X'\in\mathcal{P}_{i,V}$ is a representation such that $X_1'=X_1$ and $X_2'\neq X_2$. Then the coefficients corresponding to $X'$ in the expansions of $(D_{i+1}\circ U_i)(X)$ and $(U_{i-1}\circ D_i)(X)$ are zero unless $\operatorname{dim}_k(X_2\cap X_2')=d_2-1$ which is equivalent to $\operatorname{dim}_k(X_2+X_2')=d_2+1$. In this case, the coefficient corresponding to $X'$ in the expansion of $(D_{i+1}\circ U_i-U_{i-1}\circ D_i)(X)$ is equal to $1-1=0$. The fourth case where $X'\in\mathcal{P}_{i,V}$ is a representation such that $X_2'=X_2$ and $X_1'\neq X_1$ is treated similarly. Fifth, suppose that $X'=X$. The coefficient corresponding to $X'$ in the expansion of $(D_{i+1}\circ U_i)(X)$ is equal to the number of $(d_1+1)$-dimensional vector spaces $Y_1$ such that $X_1\subseteq Y_1\subseteq X_2$ plus the number of $(d_2+1)$-dimensional vector spaces $Y_2$ such that $X_2\subseteq Y_2\subseteq k^a$, namely $[d_2-d_1]_q+[a-d_2]_q$. Moreover, the coefficient corresponding to $X'$ in the expansion of $(U_{i-1}\circ D_i)(X)$ is equal to the number of $(d_1-1)$-dimensional vector spaces $Z_1$ such that $0\subseteq Z_1\subset X_1$ plus the number of $(d_2-1)$-dimensional vector spaces $Z_2$ such that $X_1\subseteq Z_2\subseteq X_2$, namely $[d_1]_q+[d_2-d_1]_q$. The difference of the two terms is equal to the term $[a-d_2]_q-[d_1]_q$ from equation (\ref{eqn:commutator}). For all other cases $X'$, the coefficient corresponding to $X'\in\mathcal{P}_{i,V}$ in the expansion of $(D_{i+1}\circ U_i-U_{i-1}\circ D_i)(X)$ is equal to $0$.

  Next, we claim that if $0\leq i< a$, then $U_i$ is injective. Dually, if $a<i\leq 2a$, then $D_i$ is injective. We prove the first statement only, the second statement can be proved using the same arguments. The assumption $i<a$ implies $0<a-i=a-d_2-d_1$ for every $X\in\mathcal{P}_{i,V}$. It follows that $a-d_2>d_1$ and $[a-d_2]_q>[d_1]_q$ for every $X\in\mathcal{P}_{i,V}$. By equation (\ref{eqn:commutator}) the matrix of the linear map $D_{i+1}\circ U_i-U_{i-1}\circ D_i$ (with respect to the standard basis) is a diagonal matrix with positive diagonal entries. Especially, it is positive definite. Moreover, the adjointness of $U_{i-1}$ and $D_i$ implies that the matrix corresponding to the composition $U_{i-1}\circ D_i$ (with respect to the standard basis) is positive semi-definite. The sum of the two matrices must be positive definite as well. Especially, $D_{i+1}\circ U_i$ must be invertible so that $U_i$ is injective.
  \end{proof}

  \begin{ex} Figure \ref{fig:SubrepPoset} shows the poset of subrepresentations of the $4$-dimensional representation $P_1^2$ over the finite field with $5$ elements. The subrepresentations of dimension $2$ form a maximum antichain of cardinality $7$.
  \end{ex}

  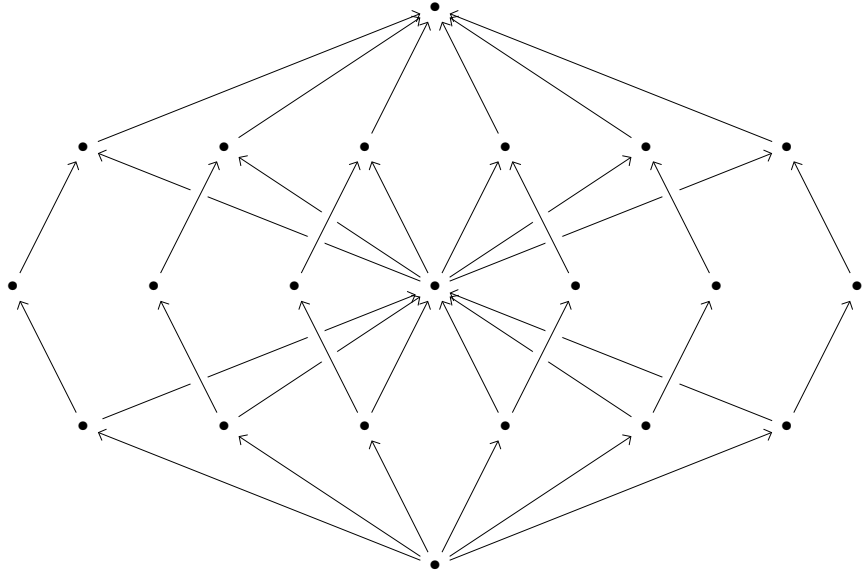
\begin{figure}
      \centering
      \resizebox{.7\linewidth}{!}{
      \begin{tikzpicture}

          \newcommand{\hor}{2cm}
          \newcommand{\ver}{2cm}

          \node (0) at (0,0)                  {$\bullet$};
          \node (1) at (-2.5*\ver,\hor)       {$\bullet$};
          \node (2) at (-1.5*\ver,\hor)       {$\bullet$};
          \node (3) at (-0.5*\ver,\hor)       {$\bullet$};
          \node (4) at (0.5*\ver,\hor)        {$\bullet$};
          \node (5) at (1.5*\ver,\hor)        {$\bullet$};
          \node (6) at (2.5*\ver,\hor)        {$\bullet$};
          \node (7) at (0,2*\hor)             {$\bullet$};
          \node (8) at (-3*\ver,2*\hor)       {$\bullet$};
          \node (9) at (-2*\ver,2*\hor)       {$\bullet$};
          \node (10) at (-1*\ver,2*\hor)      {$\bullet$};
          \node (11) at (1*\ver,2*\hor)       {$\bullet$};
          \node (12) at (2*\ver,2*\hor)       {$\bullet$};
          \node (13) at (3*\ver,2*\hor)       {$\bullet$};
          \node (14) at (-2.5*\ver,3*\hor)    {$\bullet$};
          \node (15) at (-1.5*\ver,3*\hor)    {$\bullet$};
          \node (16) at (-0.5*\ver,3*\hor)    {$\bullet$};
          \node (17) at (0.5*\ver,3*\hor)     {$\bullet$};
          \node (18) at (1.5*\ver,3*\hor)     {$\bullet$};
          \node (19) at (2.5*\ver,3*\hor)     {$\bullet$};
          \node (20) at (0,4*\hor)            {$\bullet$};

          \draw[-angle 90,relative] (0) to (1);
          \draw[-angle 90,relative] (0) to (2);
          \draw[-angle 90,relative] (0) to (3);
          \draw[-angle 90,relative] (0) to (4);
          \draw[-angle 90,relative] (0) to (5);
          \draw[-angle 90,relative] (0) to (6);

          \draw[-angle 90,relative,shorten >=2pt,shorten <=2pt] (1) to (7);
          \draw[fill=white,draw=white] (-3.35,2.65) circle [radius=0.125];
          \draw[fill=white,draw=white] (-1.7,3.35) circle [radius=0.125];
          \draw[-angle 90,relative] (2) to (7);
          \draw[fill=white,draw=white] (-1.5,3) circle [radius=0.125];
          \draw[-angle 90,relative] (3) to (7);
          \draw[-angle 90,relative] (4) to (7);
          \draw[-angle 90,relative] (5) to (7);
          \draw[fill=white,draw=white] (1.5,3) circle [radius=0.125];
          \draw[-angle 90,relative] (6) to (7);
          \draw[fill=white,draw=white] (3.35,2.65) circle [radius=0.125];
          \draw[fill=white,draw=white] (1.7,3.35) circle [radius=0.125];

   		    \draw[-angle 90,relative] (1) to (8);
          \draw[-angle 90,relative] (2) to (9);
          \draw[-angle 90,relative] (3) to (10);
          \draw[-angle 90,relative] (4) to (11);
          \draw[-angle 90,relative] (5) to (12);
          \draw[-angle 90,relative] (6) to (13);

          \draw[-angle 90,relative] (7) to (14);
          \draw[-angle 90,relative] (7) to (15);
          \draw[fill=white,draw=white] (-3.35,5.35) circle [radius=0.125];
          \draw[fill=white,draw=white] (-1.7,4.65) circle [radius=0.125];
          \draw[-angle 90,relative] (7) to (16);
          \draw[fill=white,draw=white] (-1.5,5) circle [radius=0.125];
          \draw[-angle 90,relative] (7) to (17);
          \draw[-angle 90,relative] (7) to (18);
          \draw[fill=white,draw=white] (1.5,5) circle [radius=0.125];
          \draw[-angle 90,relative] (7) to (19);
          \draw[fill=white,draw=white] (3.35,5.35) circle [radius=0.125];
          \draw[fill=white,draw=white] (1.7,4.65) circle [radius=0.125];

          \draw[-angle 90,relative] (8) to (14);
          \draw[-angle 90,relative] (9) to (15);
          \draw[-angle 90,relative] (10) to (16);
          \draw[-angle 90,relative] (11) to (17);
          \draw[-angle 90,relative] (12) to (18);
          \draw[-angle 90,relative] (13) to (19);

          \draw[-angle 90,relative] (14) to (20);
          \draw[-angle 90,relative] (15) to (20);
          \draw[-angle 90,relative] (16) to (20);
          \draw[-angle 90,relative] (17) to (20);
          \draw[-angle 90,relative] (18) to (20);
          \draw[-angle 90,relative] (19) to (20);

      \end{tikzpicture}}

      \caption{Subrepresentation poset in type $A_2$}
      \label{fig:SubrepPoset}
  \end{figure}

  \subsection{A Sperner result for subrepresentation posets over pointed sets}

  A \textit{pointed set} $(A,0_A)$ is a set $A$ together with a distinguished element $0_A\in A$. The set is also known as a \textit{based set} with \textit{basepoint} $0_A$. A \textit{pointed subset} of a pointed set $(A,0_A)$ is a subset $B\subseteq A$ that contains $0_A$ as its basepoint. Let $(A,0_A)$ and $(B,0_B)$ be two pointed sets. A morphism between $(A,0_A)$ and $(B,0_B)$ is a map $f\colon A\to B$ with $f(0_A)=0_B$ such that the restriction $f\vert_{A\setminus f^{-1}(O_B)}\colon A\setminus f^{-1}(O_B) \to B$ is injective. The pointed sets together with the morphisms form the \textit{category of pointed sets}. Some authors also use the term \textit{vector spaces over the field with one element}. Especially, we say that the morphism $f\colon (A,0_A)\to(B,0_B)$ is an \textit{isomorphism} if there exist a morphism $g\colon (B,0_B)\to (A,0_A)$ such that $g\circ f =\id_{A}$ and $f\circ g=\id_B$. Note that in this case $f$ and $g$ are bijections and we say that $(A,0_A)$ and $(B,0_B)$ are \textit{isomorphic} and we write \linebreak$(A,0_A)\cong(B,0_B)$. The \textit{direct sum} $(A,0_A)\oplus(B_,0_B)$ of two pointed sets is the set \[C=(A\setminus\{0_A\})\sqcup(B\setminus\{0_B\})\sqcup\{0_C\}\] with basepoint $0_C$. We have canonical inclusions $\iota_A\colon A\to C$ and $\iota_B\colon B\to C$ and canonical projections \linebreak$\pi_A\colon C\to A$ and $\pi_B\colon C\to B$ such that $\pi_A\circ \iota_A=\id_A$ and $\pi_B\circ\iota_B=\id_B$. Moreover, we define the \textit{dimension} $\operatorname{dim}(A,0_A)=\vert A\vert$ of a pointed set to be the cardinality of the underlying set. Up to isomorphism there is exactly one pointed set of dimension $0$ which we sometimes abbreviate as $0$.

  Let $Q=(Q_0,Q_1)$ be a quiver. A representation of $Q$ with values in the category of pointed sets is a collection $X=((X_i,0_{X_i})_{i\in Q_0},(X_{\alpha})_{\alpha\in Q_1})$ consisting of a pointed set $(X_i,0_{X_i})$ for every vertex $i\in Q_0$ and a morphism of pointed sets $X_{\alpha}\colon (X_i,0_{X_i})\to (X_j,0_{X_j})$ for every arrow $\alpha\colon i\to j$ in $Q_1$. Szczesny~\cite{Sc} studies such representations. Let us recall some basic constructions. First, there is the zero representation with $X_i=0$ for all vertices $i\in Q_0$. A \textit{subrepresentation} $U$ of $X$ is representation of $Q$ such that \linebreak$(U_i,0_{U_i})\subseteq (X_i,0_{X_i})$ is a pointed subset for every vertex $i\in Q_i$ and $U_{\alpha}(u)=V_{\alpha}(u)$ for every arrow $\alpha\colon i\to j$ in $Q_1$ and every element $u\in U_i$. Especially we have $U_{\alpha}(U_i)\subseteq U_j$ for every arrow $\alpha$. A representation is called \textit{simple} if it does not admit a nonzero proper subrepresentation. Suppose that $X,Y$ are two representations of the same quiver $Q$ with values in the category of pointed sets. A \textit{morphism} $\phi\colon X\to Y$ is a collection of morphisms of pointed sets $\phi_i\colon (X_i,0_{X_i})\to (Y_i,0_{Y_i})$ for all vertices $i\in Q_0$ such that $X_{\alpha}\circ\phi_i=\phi_j\circ Y_{\alpha}$ for all arrows $\alpha\colon i\to j$ in $Q_1$. It is called an \textit{isomorphism} if every $X_{\alpha}$ is an isomorphism in the category of pointed sets. In this case we say that $X$ and $Y$ are \textit{isomorphic} and we write $X\cong Y$.

  The \textit{direct sum} $X\oplus Y$ is the representation with $(X\oplus Y)_i=(X_i,0_{X_i})\oplus (Y_i,0_{Y_i})$ for all vertices $i\in Q_0$ and $(X\oplus Y)_{\alpha}(x)=X_{\alpha}(x)$ and $(X\oplus Y)_{\alpha}(y)=Y_{\alpha}(y)$ for all elements $x\in X_i\setminus\{0_{X_i}\}$, $y\in Y_i\setminus\{0_{Y_i}\}$ and all arrows $\alpha\colon i\to j$ in $Q_1$. A representation is called \textit{decomposable} if it is isomorphic to a direct sum $X\oplus Y$ with \linebreak$X,Y\neq 0$. It is called \textit{indecomposable} otherwise. Szczesny \cite[Theorem 5]{Sc} describes indecomposable representations with values in the category of pointed sets for quivers whose underlying undirected diagram is a tree. In this case, the isomorphism classes of indecomposable representations correspond to connected subquivers $Q'$ of $Q$. The indecomposable representation $X'$ corresponding to $Q'$ satisfies $\operatorname{dim}(X_i,0_{X_i})=1$ if $i\in Q'$ and $\operatorname{dim}(X_i,0_{X_i})=0$ otherwise.

  Let  $X$ be a representation of $Q$ with values in the category of pointed sets. We denote by $\mathcal{P}_X$ the finite set of all subrepresentations $U$ of $X$. As before, we define a partial order on the set $\mathcal{P}_X$ by putting $U_1\leq U_2$ if and only if $U_1$ is a subrepresentation of $U_2$. As in Subsection \ref{sub:SpernerLinearSub} the poset is graded when we define a degree map $\operatorname{dim}\colon\mathcal{P}_X\to\mathbb{N}$ by $\operatorname{dim}(X)=\sum_{i\in Q_0}\operatorname{dim}(X_i)$. The following proposition is immediate:

  \begin{prop}
    Let $X$ and $Y$ be two representations of $Q$. Then the poset $(\mathcal{P}_{X\oplus Y},\leq)$ is isomorphic to the direct product $(\mathcal{P}_X,\leq)\times(\mathcal{P}_Y,\leq)$.
  \end{prop}

  \begin{prop}
    Let $r\geq 1$ and $\ell_1,\ell_2,\ldots,\ell_r\geq 1$ be integers. We consider the quiver $Q$ with underlying undirected star-shaped graph as in Subsection~\ref{subsec:quivrep}, with a unique source $c$ and $r$ rays \[c\to v_{1,i}\to v_{2,i}\to \ldots\to v_{i,\ell_{i}}\] for $1\leq i\leq r$. Let $X$ be an indecomposable representation of $Q$ with values in the category of pointed sets. Then $(\mathcal{P}_X,\leq)$ is Sperner.
  \end{prop}

  \begin{proof}
    Without loss of generality we may assume $\operatorname{supp}(X)=Q_0$. The case $r=1$ is easy so let us assume $r\geq 2$. A maximum antichain can not contain the representation $X$ itself. But $(\mathcal{P}_X\setminus\{X\},\leq)$ is isomorphic to the chain product $Ch(\ell_1,\ell_2,\ldots,\ell_r)$ from Example \ref{ex:chainprod}.
  \end{proof}

  A similar statement is true for the star-shaped quiver whose vertices are oriented towards the inner vertex.

\end{document}